\documentclass[11pt]{amsart}

  \usepackage{amssymb,amsmath,amscd,amsthm,enumerate,hyperref,verbatim}
  \usepackage[utf8]{inputenc}
  \usepackage[mathscr]{euscript}
  \usepackage{xcolor}
 \usepackage[left=1in,right=1in,top=1in,bottom=1in]{geometry}

  \DeclareMathOperator{\sgn}{sgn}
  \DeclareMathOperator{\dist}{dist}
  \allowdisplaybreaks

  \newcommand{\bbN}{{\mathbb{N}}}
  \newcommand{\bbR}{{\mathbb{R}}}
  
  \newcommand{\bbZ}{{\mathbb{Z}}}
  \newcommand{\bbC}{{\mathbb{C}}}
  
  \newcommand{\bbT}{{\mathbb{T}}}

  \newcommand{\cD}{{\mathcal{D}}}

  \newcommand{\cP}{{\mathcal{P}}}
  
  \newcommand{\cR}{{\mathcal{R}}}
  \newcommand{\cS}{{\mathcal{S}}}

  \newcommand{\norm}[1]{\left\lVert#1\right\rVert}
  \newcommand{\Z}{\ensuremath{\mathbb{Z}}}
  \newcommand{\R}{\ensuremath{\mathbb{R}}}
  
  \newcommand{\C}{\ensuremath{\mathbb{C}}}
  \newcommand{\N}{\ensuremath{\mathbb{N}}}

  \newcommand{\dd}{\ \mathrm{d}}

  \renewcommand{\Im}{\operatorname{Im}}
  \renewcommand{\Re}{\operatorname{Re}}

\newcommand{\ac}{\mathrm{ac}}
\newcommand{\loc}{\mathrm{loc}}
\newcommand{\AC}{\mathrm{AC}}

  \theoremstyle{plain}
  \newtheorem{thm}{Theorem}[section]
  \newtheorem*{thm*}{Theorem}
  \newtheorem{lem}[thm]{Lemma}
  \newtheorem{prop}[thm]{Proposition}
  \newtheorem{cor}[thm]{Corollary}

  \theoremstyle{definition}

  \newtheorem{rem}{Remark}[section]
  
  \newtheorem*{ack}{Acknowledgments}

  \numberwithin{equation}{section}
  \numberwithin{thm}{section}
  \numberwithin{defn}{section}

\begin{document}

\author[M. Luki\'c, G. Young]{Milivoje Luki\'c, Giorgio Young}
\address{Department of Mathematics, Rice University, Houston, TX~77005, USA}
\email{milivoje.lukic@rice.edu,gfy1@rice.edu}

\thanks{M.L.\ was supported in part by NSF grant DMS--1700179.}
\thanks{G.Y.\ was supported in part by NSF grant DMS--1745670.}

\title{Uniqueness of solutions of the KdV-hierarchy via Dubrovin-type flows}

\begin{abstract}
 We consider the Cauchy problem for the KdV hierarchy -- a family of integrable PDEs with a Lax pair representation involving one-dimensional Schr\"odinger operators -- under a local in time boundedness assumption on the solution.
 For reflectionless initial data, we prove that the solution stays reflectionless. For almost periodic initial data with absolutely continuous spectrum, we prove that under Craig-type conditions on the spectrum, Dirichlet data evolve according to a Lipschitz Dubrovin-type flow, so the solution is uniquely recovered by a trace formula. This applies to algebro-geometric (finite gap) solutions; more notably, we prove that it applies to small quasiperiodic initial data with analytic sampling functions and Diophantine frequency.
This also gives a uniqueness result for the Cauchy problem on the line for periodic initial data, even in the absence of Craig-type conditions.
\end{abstract}

\maketitle

\section{Introduction and Main Results}

The KdV equation
  \begin{equation}\label{eqn:KdV1}
   \partial_t q +\frac14 \partial_x^3 q -\frac32 q \partial_x q=0
  \end{equation}
was studied in the 19th century by Boussinesq \cite{Bou} and Korteweg--de Vries \cite{KdV} as a model for solitary wave phenomena. Its modern theory began in the 1960s with works of Gardner--Greene--Kruskal--Miura \cite{GGKM} and Lax \cite{L68}, who discovered that the KdV equation has infinitely many conserved quantities and can be formally rephrased in a Lax pair representation and studied through the associated family of one-dimensional Schr\"odinger operators. The conserved quantities formally correspond to mutually commuting Hamiltonians; this makes the KdV equation one of an infinite family of mutually commuting partial differential equations, called the KdV hierarchy.

The KdV hierarchy is most concisely introduced as the sequence of partial differential equations
  \begin{equation}
\label{eq:kdvn}
\frac{\partial q}{\partial t} = 2\frac{\partial \hat{f}_{n+1}}{\partial x}, \;\; (x,t) \in \R \times \R\\
\end{equation}
where the $\hat{f}_{n}$ are differential polynomials in $q$, defined recursively by
  \begin{align*}
    &\hat{f}_0=1,\;\;\hat{f}_1=q/2,\\
    &\hat{f}_{\ell +1}=-\frac12\sum_{k=1}^\ell \hat{f}_k\hat{f}_{\ell+1-k}+\frac12
    \sum_{k=0}^\ell \left( q\hat{f}_k\hat{f}_{\ell -k}+\frac14\frac{\partial \hat{f}_k}{\partial x}
    \frac{\partial \hat{f}_{\ell-k}}{\partial x}-\frac12 \frac{\partial^2\hat{f}_k}{\partial x^2}\hat{f}_{\ell-k} \right).
  \end{align*}
  Note that setting $n=0$ in \eqref{eq:kdvn} gives the translation flow $\partial_t q = \partial_x q$ and setting $n=1$ recovers the KdV equation \eqref{eqn:KdV1}. For $n\ge 1$, the $n$-th equation is a nonlinear partial differential equation with spatial derivatives up to order $2n+1$, called the KdV-$n$ equation. Due to the Lax pair representation, these PDEs are said to be integrable; 50 years after this discovery, many other integrable PDEs have been discovered, and KdV continues to be one of the central models on which new integrability phenomena are explored \cite{De3}.

Of course, rigorous results about the KdV equation and hierarchy are dependent on the type of initial data being considered,
\begin{equation} \label{eq:initial}
  q(x,0)=V(x).
  \end{equation}
The first application of the Lax pair representation was to rapidly decaying initial data via the inverse scattering transform of Gelfand--Levitan--Marchenko. Other classical applications were to periodic and  finite gap quasiperiodic initial data; in those applications, solutions can be parametrized by Dirichlet data, which evolve by Dubrovin-type flows, or by an Abel map which linearizes the trajectories on a torus.
In particular, the finite gap quasiperiodic solutions are described by an algebro-geometric approach associated
to a compact Riemann surface. The solutions are quasiperiodic in both space and time because they can be expressed in the form $q(x,t) = F(\delta x + \zeta t)$ where $F:\bbT^d \to \bbR$ is continuous and $\delta,\zeta \in \bbR^d$, where $d< \infty$ is the number of gaps.

These theories motivate a more general consideration of the KdV hierarchy with almost periodic initial data. Recall that a function $f:\bbR \to \bbC$ is called almost periodic if the set of its translates $\{f(\cdot - s) \mid s \in\bbR\}$ is precompact in the uniform norm; this includes continuous periodic and quasiperiodic functions as special cases. In the study of the KdV hierarchy with almost periodic initial data, one of the new difficulties is in the nature of the conserved quantities which take the form of spatial averages over $\bbR$ such as $\lim_{L\to\infty} \frac 1{2L} \int_{-L}^L q(x,t)^2 \,dx$, which are not useful for obtaining local control over the solution. Indeed, even short time existence of solutions is not known for arbitrary almost periodic initial data. From a spectral/scattering perspective, the difficulty comes from the fact that almost periodic Schr\"odinger operators have very diverse spectral properties.

Despite these obstacles, the KdV hierarchy with almost periodic initial data is an active area of research \cite{Eg1,Eg2,Ts12,DG2,BDGL18,EVY18}, motivated by a question of Deift \cite{De,De2} about whether such solutions are almost periodic in time. Most of these works are based on the reflectionless property and on spectral theoretic techniques initially developed in the time-independent setting (i.e. for the translation flow) \cite{C89,SY1,SY2}. Analogous questions have also been studied in the setting of the nonlinear Schr\"odinger equation and the Toda lattice \cite{Oh,BdME,VY,BDLV}.

To get more precise, we must recall some basic facts about Schr\"odinger operators, starting with a time-independent setting. We denote by $\AC_\loc(\bbR)$ the set of functions on $\bbR$ which are absolutely continuous on every compact interval in $\bbR$; such functions have a derivative in $L^1_\loc(\bbR)$. We can describe in those terms the familiar Sobolev space
\[
W^{2,2}(\bbR) = \{ f \in L^2(\bbR) \mid f \in \AC_\loc(\bbR), f'\in \AC_\loc(\bbR), f'' \in L^2(\bbR) \}.
\]
For a bounded function $W: \bbR \to \bbR$, we consider the Schr\"odinger operator
  \begin{align}
  \label{eq:Schrod}
  H_{W}=-\frac{\dd^2}{\dd x^2}+W(x),
  \end{align}
which is an unbounded self-adjoint operator on $L^2(\bbR)$ with the domain $W^{2,2}(\bbR)$.
In particular, $\sigma(H_W) \subset \bbR$. For any $z\in \bbC \setminus \sigma(H_W)$, the resolvent operator $(H_W - z)^{-1}$ exists and is an integral operator; its integral kernel is called the Green's function and denoted $G(x,y;z,W)$. Of particular interest is the diagonal Green's function, obtained by setting $y=x$.
This is a Herglotz function (an analytic function which maps $\bbC_+$ to itself)
so it has nontangential boundary values Lebesgue-a.e. on $\bbR$. The Schr\"odinger operator $H_W$ is called reflectionless if
\begin{align} \label{eq:reflectionlessness}
 \Re G(x,x;\lambda+i0, W)=0,\;\;\text{for Lebesgue a.e.}\;\;\lambda\in\sigma(H_W).
 \end{align}
Let us denote by $\cR(S)$ the set of bounded
potentials $W$ such that $H_W$ is reflectionless and $\sigma(H_W) = S$;
we make $\cR(S)$ a metric space with the $L^\infty$-metric.

Our first result is that under a boundedness assumption on solutions to KdV-$n$, the reflectionless property is preserved in time; this generalizes or extends previous results for the KdV equation in \cite{R08,Re,BDGL18}.

\begin{prop}
\label{reflectionlessness}
Let $q(x,t)$ be a classical solution to the Cauchy problem \eqref{eq:kdvn}, \eqref{eq:initial} on $x\in \bbR$, $t\in [0,T]$
for some $T < \infty$,
obeying the boundedness condition
\begin{equation}\label{21jul3}
q,\partial_{x}^{2n}q\in L^\infty(\R\times [0,T]).
\end{equation}
Let $S = \sigma(H_V)$. If $V\in \cR(S)$, then $q(\cdot,t)\in \cR(S)$ for all $t\in [0,T]$.
\end{prop}

In the above, we consider classical solutions $q$ of \eqref{eq:kdvn},
i.e., solutions such that $q$ is $2n+1$ times differentiable in $x$,
$\partial_x^{2n+1} q$ is jointly continuous in $x$ and $t$, $q$ is differentiable in $t$ and \eqref{eq:kdvn} holds pointwise.

Periodic potentials are usually treated through a well-posedness theory in Sobolev spaces $H^\beta(\bbT)$, and described in terms of a nonlinear evolution map $\cS_t:H^\beta(\bbT) \to H^\beta(\bbT)$ which is a homeomorphism for each $t$ and generates solutions $q(\cdot,t) = \cS_t(V)$ of KdV-$n$. Built into that formulation is the a priori assumption that solutions $q(x,t)$ stay periodic in $x$ for all $t$. In contrast, Prop.~\ref{reflectionlessness} allows us to consider locally bounded solutions and actually recover spatial periodicity for periodic intial data. Indeed, we have the following corollary as a consequence of Prop.~\ref{reflectionlessness}.

\begin{cor}
  \label{periodicity}
  Let $q(x,t)$ be a classical solution to the Cauchy problem \eqref{eq:kdvn}, \eqref{eq:initial} on $x\in \bbR$, $t\in [0,T]$
  for some $T < \infty$,
  obeying the boundedness condition \eqref{21jul3}.
  Suppose that the initial data satisfies $V(x+\ell)=V(x)$ for an $\ell > 0$ and all $x\in \R$.
  Then, for all $t\in [0,T]$, $x\in \R$,
  $q(x+\ell,t)=q(x,t)$ and $q(\cdot, t) = \cS_t(V)$.
\end{cor}

Note that since our result is about uniqueness, finiteness of $T$ in the previous theorem should be viewed as a strength, not a weakness;  it immediately implies the following result for global solutions.

\begin{cor}
  Let $q(x,t)$ be a classical solution to the Cauchy problem \eqref{eq:kdvn}, \eqref{eq:initial} on $x\in \bbR$, $t \in [0, \infty)$,
  obeying the boundedness condition \eqref{21jul3} for all $T < \infty$.
  Suppose that the initial data satisfies $V(x+\ell)=V(x)$ for an $\ell> 0$
  and all $x\in \R$.
  Then, for all $t\in [0,\infty)$, $x\in \R$,
  $q(x+\ell,t)=q(x,t)$ and $q(\cdot, t) = \cS_t(V)$.
\end{cor}

After these results for the periodic problem,
we will return to the more general setting, inspired by the
regime of almost periodic initial data. We recall an important
link between the reflectionless property and almost periodicity: by Kotani theory \cite{Kotani}, if $W$ is almost periodic and its absolutely continuous spectrum is $\sigma_\ac(H_W) = \sigma(H_W)$, then $H_W$ is reflectionless.
We will use this in an application of our results in the final section of the paper.

In the nonperiodic setting, our work will use certain thickness conditions on the spectrum. By general principles, $S = \sigma(H_W)$ is a closed set bounded below but not bounded above, so it can be written in the form
\begin{equation}\label{eqn:setS}
S = [ \underline{E}, \infty) \setminus \cup_{j\in J} (E_j^-, E_j^+),
\end{equation}
where $\underline{E} = \inf S$, $J$ is a countable indexing set and $(E_j^-, E_j^+)$ denote maximal open intervals in $ [ \underline{E}, \infty) \setminus S$ called gaps. The ``finite gap" solutions correspond to sets $S$ with finitely many gaps, and periodic potentials have a sequence of isolated gaps indexed by $n\in\bbN$ such that $E_n^\pm \to +\infty$ as $n\to\infty$; however, for general almost periodic initial data the set $S$ is generically a Cantor set and can even have zero Lebesgue measure \cite{Av,DFL17}. Many difficulties arise from the possible accumulation of gaps at finite points.

Using the observation that $G(x,x;z,W)$ is continuous and strictly increasing for $z\in (E_j^-,E_j^+)$, Craig~\cite{C89} introduced Dirichlet data for $W \in \cR(S)$ by
  \begin{align}
    \label{eq:Dirichleteig}
    \mu_j(x) =\begin{cases}
    z\in (E_j^-,E_j^+), & G(x,x;z,W)=0, \\
    E_j^+, & G(x,x;z,W)<0\;\;\forall z\in (E_j^-,E_j^+), \\
    E_j^-, & G(x,x;z,W)>0\;\;\forall z\in (E_j^-,E_j^+).
    \end{cases}
  \end{align}
If $\mu_j(x) \in (E_j^-, E_j^+)$, additional information is in the value of $\sigma_j(x) = - \partial_x G(x,x;z,W)\vert_{z=\mu_j(x)}  \in \{\pm 1\}$. Together, for fixed $x\in\bbR$, $\mu_j(x)$ and $\sigma_j(x)$ can be thought of as lying on a double cover of the interval $[E_j^-,E_j^+]$ with gap edges identified, and they can be combined into an angular variable $\varphi_j(x) \in \bbR / 2\pi \bbZ$ defined by the conditions
  \begin{align}
    \label{eq:coordinates1}
    \mu_j(x) &=E_{j}^-+(E_j^+-E_j^-)\cos^2\left( \frac{\varphi_j(x)}{2}\right), \\
    \label{eq:coordinates2}
    \sigma_j(x) &= -\sgn \sin \varphi_j(x).
  \end{align}
We view this as a correspondence from $W \in \cR(S)$ to a trajectory $\varphi(x) = (\varphi_j(x))_{j\in J}$ on the ``torus of Dirichlet data" $\cD(S) := \bbT^J$.

In particular, the map $W \mapsto \varphi(0)$ is a map from $\cR(S)$ to $\cD(S)$, which will be a  homeomorphism in our regime. The recovery of $W$ from its Dirichlet data $\varphi(0)$
uses the behavior of Dirichlet data with respect to translation: it was proven
in \cite{C89,BDGL18}
that the translation flow is governed
by the ordinary differential equation
  \begin{equation}\label{eq:translationflow0}
    \partial_x\varphi(x)=\Psi(\varphi(x)),
  \end{equation}
  where $\Psi$ is the Dubrovin-type vector field on $\cD(S)$ with components
    \begin{align}
    \label{eq:Psidef}
    \Psi_j(\varphi)=2\sqrt{(\mu_j -\underline{E}) \prod_{\ell\ne j}\frac{(E_\ell^{-}-\mu_j)(E_{\ell}^+-\mu_j)}{(\mu_\ell -\mu_j)^2}}.
  \end{align}
Under suitable thickness assumptions on the spectrum, which will be discussed below, Craig \cite{C89} proved that $\Psi$ is Lipschitz, so the solution to the ODE \eqref{eq:translationflow0} is unique for any given initial data $\varphi(0) \in \cD(S)$. This trajectory $\varphi(x)$ determines the potential by \eqref{eq:coordinates1} and the trace formula
\begin{equation}\label{eq:trace9}
W(x) = \underline{E}+\sum_{j\in J}(E_j^{-} + E_j^{+}  - 2\mu_j(x) )
\end{equation}
which holds in great generality \cite{GHSZ}. In summary, the map from $\cR(S)$ to $\cD(S)$ given by $W \mapsto \varphi(0)$ is then a homeomorphism, with an explicit inverse given by solving the ODE \eqref{eq:translationflow0} and applying the trace formula \eqref{eq:trace9}.

For the set $S$ with gaps denoted as in \eqref{eqn:setS}, we denote
  \begin{align}
    \gamma_j&=E_j^{+}-E_{j}^{-} \nonumber, \\
    \eta_{j,l}&=\dist((E_j^{-},E_j^{+}),(E_l^{-},E_l^{+}))\nonumber, \\
    \eta_{j,0}&=\dist(\underline{E},(E_j^{-},E_j^{+}))\nonumber, \\
    C_j&=(\eta_{j,0}+\gamma_j)^{1/2}\prod_{\substack{l\in J\\\ell\ne j}}\left( 1+\frac{\gamma_l}{\eta_{j,l}} \right)^{1/2} \label{Cjdefn}.
  \end{align}
To study the KdV-$n$ equation, we fix $n$ and assume that $S$ obeys the moment condition
      \begin{align}
      \label{eq:Craig5}
      \sum_{k\in J}\gamma_{k}(1+\eta_{k,0}^n)<\infty,
    \end{align}
 and we set the metric on $\cD(S)$ given by
    \begin{align}
    \label{eq:metric}
    \norm{\varphi -\tilde{\varphi}}_{\mathcal{D}(S)}=\sup_{j\in J}\gamma_j^{1/2}(1+\eta_{j,0}^n)^{1/2}\norm{\varphi_j-\tilde{\varphi}_j}_{\mathbb{T}}
  \end{align}
which is consistent with the product topology. In particular,  the torus $\cD(S)$ is compact.

    We consider the scalar fields on $\cD(S)$
  \begin{align}
    \label{eq:higherordertrace}
  Q_k=\underline{E}^k + \sum_{j\in J} \left( (E_j^{-})^k+(E_j^{+})^k-2 \mu_j^k \right)
  \end{align}
and
  \begin{equation}\label{eq:Rmdefn}
  R_m = \sum_{\substack{\alpha\in \N_0^n\\\sum_{k=1}^m k\alpha_k=m}}\prod_{k=1}^m \frac{Q_k^{\alpha_k}}{\alpha_k!(2k)^{\alpha_k}}.
  \end{equation}
It will follow from \eqref{eq:Craig5} that $Q_k$ and thus $R_m$ are continuous scalar fields for all $k, m\le n$.

We also consider another Dubrovin-type vector field $\Xi$ with components
  \begin{align}
    \label{eq:Xidef}
    \Xi_j(\varphi)=\left( \sum_{\ell=0}^n R_{n-\ell}\mu_j^\ell \right)\Psi_j.
  \end{align}
We formulate conditions to ensure that the vector field $\Xi$ is Lipschitz on $\cD(S)$. These Craig-type conditions are as follows,
    \begin{align}
      \label{eq:Craig1}
      \sum_{k\in J}\gamma_{k}^{1/2}(1+\eta_{k,0}^n)^{1/2} <\infty,
    \end{align}
    \begin{align}
      \label{eq:Craig2}
      \sup_{j\in J}C_j(1+\eta_{j,0}^n)^{3/2}
      \sum_{k\ne j}\frac{\gamma_k^{1/2}\gamma_j^{1/2}}{\eta_{j,k}}<\infty,
    \end{align}
    \begin{align}
      \label{eq:Craig3}
      \sup_{j\in J}\frac{\gamma_j(1+\eta_{j,0}^n)C_j}{\eta_{j,0}}<\infty,
    \end{align}
    \begin{align}
      \label{eq:Craig4}
      \sup_{j\in J}C_j\gamma_{j}^{1/2}(1+\eta_{j,0}^{n})^{3/2}<\infty,
    \end{align}
with $C_j$ given by \eqref{Cjdefn}. Although spectral gaps can accumulate at a finite point,  these conditions provide quantitative bounds on the accumulation of gaps and serve as thickness conditions on the spectrum.  \eqref{eq:Craig1} is a kind of moment condition and it implies \eqref{eq:Craig5}.

As in the discussion surrounding \eqref{eq:translationflow0}, by \cite{C89,BDGL18}, the translation flow is governed by the following ordinary differential equation for fixed $t$,
  \begin{align}
    \label{eq:translationflow}
    \partial_x\varphi(x,t)=\Psi(\varphi(x,t)).
  \end{align}
  It was also proven in \cite{C89} that $\Psi$ is Lipschitz under weaker assumptions than
  \eqref{eq:Craig1}, \eqref{eq:Craig2}, \eqref{eq:Craig3}, and \eqref{eq:Craig4}.
  In this paper we will show that
  \begin{align}
    \label{eq:Kdvnflow}
    \partial_t\varphi(x,t)=\Xi(\varphi(x,t))
  \end{align}
  and that $\Xi$ is Lipschitz under assumptions
  \eqref{eq:Craig1}, \eqref{eq:Craig2}, \eqref{eq:Craig3}, and \eqref{eq:Craig4}.  Note that the Craig-type conditions \eqref{eq:Craig1}, \eqref{eq:Craig2}, \eqref{eq:Craig3}, \eqref{eq:Craig4} are inevitably stronger than those in \cite{C89} (which were designed to control the translation flow $\Psi$) or those in \cite{BDGL18} (designed to control the flow corresponding to the KdV equation, $n=1$).
 The Lipschitz property will imply that the two-parameter solution $\varphi(x,t)$ solving \eqref{eq:translationflow} and \eqref{eq:Kdvnflow} is uniquely determined by $\varphi(0,0)$, which is in turn uniquely determined by the initial data $V$. We prove the following:

  \begin{thm}
  \label{thm:Dubrovin}
  Let $q(x,t)$ be a classical solution to \eqref{eq:kdvn}, \eqref{eq:initial} on $x\in \bbR$, $t\in [0,T]$ for some $T < \infty$, with
  initial data $V$ such that $H_V$ is reflectionless and the spectrum
  $S = \sigma(H_V)$ obeys  \eqref{eq:Craig1}, \eqref{eq:Craig2}, \eqref{eq:Craig3}, \eqref{eq:Craig4}. We also assume
  $q$ obeys \eqref{21jul3}.
   Then, for all $t\in [0,T]$,
   \begin{enumerate}[(a)]
   \item The space and time evolution of Dirichlet data $\varphi = (\varphi_j)_{j\in J}$ corresponding to $q(x, t)$ is uniquely determined from $\varphi(0,0)$ as the solution of the flows \eqref{eq:translationflow}, \eqref{eq:Kdvnflow} with respect to the Lipshitz vector fields $\Psi$ and $\Xi$.
\item The solution $q(x,t)$ obeys
  \begin{equation}\label{21jul4}
    q(x,t)=\underline{E}+\sum_{j\in J} \left( E_{j}^++E_{j}^--2\mu_j(x,t) \right)
  \end{equation}
  where $\sigma(H_V)=[\underline{E},\infty)\setminus \bigcup_{j\in J}(E_j^-,E_j^+)$ and
  $\mu_j(x,t)$ is the Dirichlet data.
\item In particular, $q$ is uniquely determined by \eqref{eq:Kdvnflow} and \eqref{21jul4}.
\end{enumerate}
  \end{thm}

Again, we have the following result for global solutions.
\begin{cor}
  Let $q(x,t)$ be a classical solution to \eqref{eq:kdvn}, \eqref{eq:initial} on $x\in \bbR$, $t\in [0,\infty)$, with
  initial data $V$ such that $H_V$ is reflectionless and the spectrum $S = \sigma(H_V)$ obeys \eqref{eq:Craig1}, \eqref{eq:Craig2}, \eqref{eq:Craig3}, and \eqref{eq:Craig4}.
  If $q$ obeys the boundedness condition \eqref{21jul3} for all $T <\infty$, then the conclusions of the previous theorem apply to $q$ for all $t \in [0,\infty)$.
\end{cor}

Analogous results hold for negative time.

It should be noted that Theorem~\ref{thm:Dubrovin} doesn't assume
that our solution stays almost periodic in $x$; instead, this follows as a consequence. In this way, in addition to saying something new
for the periodic case as we highlighted in Corollary~\ref{periodicity}, we also see that our results apply to
the finite gap quasiperiodic case (where the Craig-type
conditions are trivially satisfied).
The theorem also applies to a class of small quasiperiodic
initial data which we will describe below.

An earlier uniqueness result for the KdV equation was proved by Binder--Damanik--Goldstein--Lukic \cite{BDGL18} and is based on work of Rybkin \cite{R08} on the time evolution of Weyl solutions, $m$-functions and $M$-matrices under the KdV equation. Our paper can be viewed as a generalization of these results to the entire KdV hierarchy. Moreover, through a more careful analysis of the time evolution of eigensolutions, our results improve those in \cite{R08,BDGL18} even for the KdV equation. Where earlier results for the KdV equation require $q, \partial_x^3 q \in L^\infty(\bbR \times [0,T])$, our Theorem~\ref{thm:Dubrovin} for $n=1$ only requires $q, \partial_x^2 q \in L^\infty(\bbR \times [0,T])$.

The paper \cite{BDGL18} also proved existence and almost periodicity of solutions to the KdV equation for a class of reflectionless initial data with Craig-type conditions on the spectrum $S$; these results were generalized by Eichinger--VandenBoom--Yuditskii \cite{EVY18} to a more general class of $S$ (Widom sets with the Direct Cauchy Theorem property and a moment condition) and to the entire KdV hierarchy. Our uniqueness results can be viewed as complementary to those existence and almost periodicity results. Within the scope of applicability, our results show that the almost periodic solution constructed in \cite{EVY18} is the only locally bounded solution in the sense of \eqref{21jul3}.

Finally, we describe the application to small quasiperiodic initial data. Let $\epsilon>0$, $0\leq \kappa_0\leq 1$, and $\omega\in \R^\nu$ for some $\nu\in \N$.
We say $V\in \mathcal{P}(\omega,\epsilon,\kappa_0)$ if $V:\R\to \R$ is of the form
\begin{align}
  \label{eq:Vseries}
V(x)=\sum_{m\in \N^{\nu}}c(m)e^{2\pi i m\cdot \omega x}
\end{align}
where
\[|c(m)|\leq \epsilon\exp(-\kappa_0 |m|), \;\forall n\in \Z^{\nu}.
\]

All results will be in the small coupling regime, $\epsilon < \epsilon_0(a_0,b_0,\kappa_0)$.
The direct spectral theory of $H_V$ for $\epsilon$ has been studied extensively by
Eliasson~\cite{El} and Damanik--Goldstein~\cite{DG1},
with studies of the inverse spectral theory and the KdV equation in \cite{Ts12,DGL1,DGL2,DGL3,BDGL18}. Using that theory, we will be able to prove that Theorem~\ref{thm:Dubrovin} applies to initial data $V \in \cP(\omega,\epsilon,\kappa_0)$.

We now state our uniqueness theorem as applied to data of the form \eqref{eq:Vseries} above, with $\omega$ satisfying
a Diophantine condition,
\begin{align}
\label{eq:Diophantine}
|m \cdot \omega|\geq a_0|m|^{-b_0}
\end{align}
for some $0<a_0<1$ and $\nu<b_0<\infty$.
\begin{thm}
  \label{thm:app}
Let $\omega\in \R^\nu$ obey the Diophantine condition \eqref{eq:Diophantine} for some $0<a_0<1$ and $\nu<b_0<\infty$. There is an $\epsilon_0(a_0,b_0,\kappa_0)>0$ such that
if $\epsilon<\epsilon_0$, and $V\in \mathcal{P}(\omega,\epsilon,\kappa_0)$, then
any solution of \eqref{eq:kdvn}, \eqref{eq:initial}, with $q,\partial_x^{2n}q\in L^\infty(\R\times[0,T])$
is unique.
\end{thm}

  \begin{ack}
  We would like to thank Benjamin Eichinger and Jo Nelson for clarifying conversations.
  \end{ack}

\section{Time evolution of the Weyl solutions and the Weyl $M$-matrix}

In this section, we will study the time dependence of Weyl solutions and the Weyl $M$-matrix for the family of Schr\"odinger operators associated to a fixed classical solution $q$ of \eqref{eq:kdvn}, \eqref{eq:initial} which obeys the boundedness condition \eqref{21jul3}.

Historically, the discovery of the Lax pair representation was preceded by a description in \cite{GGKM} of the time-evolution of formal eigensolutions for the Schr\"odinger operator. The main analytical result of this section is that under suitable conditions, not only are eigensolutions preserved in this way, but so are Weyl solutions, as defined below. After describing the time evolution of Weyl solutions, we will be able to compute the time evolution of Weyl $m$-functions, the Weyl $M$-matrix, and reflection coefficients.

If $W:\bbR \to \bbR$ is bounded, the Schr\"odinger operator $H_W$ is in the ``limit point" case at $\pm \infty$, i.e., for any $z\in \bbC \setminus \sigma(H_W)$ and each halfline $[0,\pm\infty)$, there is a one-dimensional subspace of solutions of
 \begin{equation}\label{jul16}
 -\psi'' + W \psi = z \psi
 \end{equation}
 which are square-integrable on that half-line. Any nontrivial eigensolution which is square-integrable on the half-line $[0,\pm\infty)$ is called a Weyl solution at $\pm\infty$ and denoted by $\psi_\pm(x;W)$; Weyl solutions are defined up to normalization. It is a common convention to set $\psi_\pm(0;W) = 1$, but this convention isn't natural in time-dependent considerations or when changing reference points.
Instead, we will assume that a Weyl solution has been chosen corresponding to $q(\cdot,0) = V$ and will consider a time evolution for the eigensolution, in which both $x$-dependence and $t$-dependence are written as a first-order (system of) ODEs.

We define, as in \cite{GH03},
  \begin{align*}
    Q(x,t)&=\begin{pmatrix}0&1\\q(x,t)-z&0\end{pmatrix},\\
        P(x,t)&=\begin{pmatrix} - \frac 12 \partial_x \hat F_n(z) & \hat F_n(z) \\ (q-z)\hat{F}_n(z)-\frac12 \partial_x^2 \hat{F}_n(z) & \frac 12 \partial_x \hat F_n(z)  \end{pmatrix},
  \end{align*}
  where
  \begin{align*}
    \hat{F}_n(z)=\sum_{\ell =0}^n\hat{f}_{n-\ell}z^{\ell}.
  \end{align*}
We consider, for any $z\in \bbC \setminus \sigma(H_V)$, the system of PDEs
  \begin{align}
    \label{eq:Weylpde}
    \begin{cases}
    \partial_t
    \nu_\pm(x,t)
    =P(x,t)\nu_\pm(x,t) \\
    \partial_x\nu_\pm(x,t)
    =Q(x,t)\nu_\pm(x,t)
    \\
    \nu_\pm(0,0)=\begin{pmatrix}\psi_\pm(0;V)\\\psi_\pm'(0;V) \end{pmatrix}
  \end{cases}
  \end{align}
  $Q(x,t)$ is the standard matrix for converting the second order eigenvalue equation,
  $H_{q(\cdot,t)}\psi=z\psi$,
  to a first order system, and $P(x,t)$ will determine the
  time evolution for the Weyl
  solutions of $H_V$.
 The matrix functions $P,Q:\R^2\to M_2(\C)$ satisfy the zero curvature condition if
  \begin{align}
    \label{eq:zerocurve}
   \partial_t Q -\partial_x P +[P,Q]=0
  \end{align}
  where $[P,Q]$ is the commutator.
 As an equality of differential expressions, the zero curvature condition \eqref{eq:zerocurve} is equivalent to \eqref{eq:kdvn} \cite{GH03}, and \eqref{eq:zerocurve} is sometimes viewed as a different way to introduce the KdV hierarchy.
  \begin{rem}
 Existence of a joint solution of the system  \eqref{eq:Weylpde} depends on the zero curvature condition  \eqref{eq:zerocurve} by standard arguments from differential geometry. While these arguments are often presented under $C^\infty$ assumptions, they hold under our smoothness conditions.
 Namely, the zero curvature condition \eqref{eq:zerocurve}
    involves  $\partial_x P$, a differential polynomial with at most $2n+1$ spatial derivatives of $q$ in each entry, and $\partial_t Q$, which includes only one time derivative of $q$ in the bottom left coordinate.
    For a classical solution $q$ of KdV-$n$, these partial derivatives exist and are jointly continuous in $(x,t)$. The system \eqref{eq:Weylpde} can be written as an autonomous system on $(x,t,\nu) \in \bbR^2 \times \bbC^2$ corresponding to vector fields $\hat{Q},\hat{P}$ written in the language of differential geometry as
    \begin{align*}
      &\hat{Q}:=\partial_x+P^{ij}(x,t)v_{j}\partial \nu_i\\
      &\hat{P}:=\partial_t+Q^{ij}(x,t)v_{j}\partial \nu_i.
    \end{align*}
These vector fields are $C^1$, since, as in the above, $P$ involves a differential polynomial with only $2n$ spatial derivatives on $q$,
while $Q$ only has $(x,t)$ dependence through $q$.  Then \eqref{eq:zerocurve} is equivalent to the vanishing of the Lie bracket $L_{\hat{P}}\hat{Q}=[\hat{P},\hat{Q}]=0$, so the flows of $\hat{P}$ and $\hat{Q}$ commute by differential geometry arguments in \cite[Lemma 5.13]{S05}. This shows existence of a solution of \eqref{eq:Weylpde}. Uniqueness follows from the fact that $\hat{Q},\hat{P}\in C^1$.
  \end{rem}

We will now prove that the first component of the solution $\nu_{\pm}(x,t)$ of \eqref{eq:Weylpde} is the Weyl solution for $H_{q(\cdot,t)}$ for each $t\in [0,T]$.

 \begin{prop}
Let $q$ be a classical solution of the Cauchy problem \eqref{eq:kdvn}, \eqref{eq:initial} obeying the boundedness condition \eqref{21jul3}. Fix $z \in \bbC \setminus \sigma(H_V)$ and let
\[
\nu_\pm(x,t) = \begin{pmatrix} \alpha_\pm(x,t) \\ \beta_\pm(x,t) \end{pmatrix}
\]
be the solution to \eqref{eq:Weylpde}. Then, for every $t\in [0,T]$, $\alpha_\pm(\cdot,t)$ is a Weyl solution at $\pm \infty$ for the potential $q(\cdot,t)$.
 \end{prop}

    \begin{proof}
    It follows from the form of $Q$ that $\partial_x\alpha_\pm(x,t) = \beta_\pm(x,t)$ and
    \[
    \partial_x^2 \alpha_\pm(x,t) = \partial_x \beta_\pm(x,t) = (q(x,t) -z)\alpha_\pm(x,t)
    \]
    so for each $t$, $\alpha_\pm(x,t)$ as a function of $x$ solves the eigensolution equation for potential $q(\cdot, t)$. It remains to prove that it is nontrivial for each $t$ and square-integrable on the corresponding half-line.

Let us denote
      \[
      g(x,t):=|\alpha_\pm(x,t)|^2+|\beta_\pm(x,t)|^2 =\norm{ {\nu}_{\pm}(x,t)}_{\C^2}^2.
      \]
      Then
      \[
      \partial_t g(x,t) =2 \Re \left\langle \partial_t {\nu}_{\pm}(x,t),
          {\nu}_{\pm}(x,t)  \right\rangle_{\C^2}  =2 \Re \left\langle P(x,t) {\nu}_{\pm}(x,t),
          {\nu}_{\pm}(x,t)  \right\rangle_{\C^2}
          \]
         so by the Cauchy-Schwarz inequality and using the operator norm of $P(x,t)$,
\[
\lvert \partial_t g(x,t) \rvert \le 2 \lVert P(x,t) \rVert g(x,t).
\]
      Since entries of $P(x,t)$ are polynomial expressions in $q, \partial_x q, \dots, \partial_x^{2n} q$, it follows from Sobolev inequalities and the boundedness assumption $q,\partial_x^{2n}q\in L^\infty(\R\times [0,T])$ that
      \[
      C = \sup_{x\in \bbR} \sup_{t\in [0,T]} \lVert P(x,t) \rVert < \infty.
      \]
Thus, it follows that
      \[
g(x,0) e^{-2C t} \le  g(x,t)\leq g(x,0)e^{2C t}
      \]
      for $t\in [0,T]$ and therefore
      \begin{equation}\label{21jul1}
e^{-2C t} \int_0^{\pm \infty} \lVert \nu_\pm(x,0) \rVert^2 \,dx \le  \int_0^{\pm \infty} \lVert \nu_\pm(x,t) \rVert^2 \,dx \le e^{2C t} \int_0^{\pm \infty} \lVert \nu_\pm(x,0) \rVert^2 \,dx.
      \end{equation}
Since $\alpha_\pm(\cdot,0)$ is a Weyl solution corresponding to the bounded potential $V$, it is by definition nontrivial. Moreover, by boundedness of $V$, the derivative of the Weyl solution is also square-integrable on the corresponding half-line \cite{Stolz92,Stolz95,Simon,LukicMMNP}, so
\[
0 < \int_0^{\pm \infty} \lVert \nu_\pm(x,0) \rVert^2 < \infty.
\]
It follows from \eqref{21jul1} that for all $t\in [0,T]$,
\[
0 < \int_0^{\pm \infty} \lVert \nu_\pm(x,t) \rVert^2 \,dx < \infty.
\]
Strict positivity implies that $\alpha_\pm(\cdot,t)$ is nontrivial and finiteness that $\alpha_\pm(\cdot,t)$ is the Weyl solution.
    \end{proof}

By the above lemma, we consider $\psi_\pm(x,t;z) = \alpha_\pm(x,t;z)$ a Weyl solution of $H(q(\cdot,t))$ at energy $z$ and we know that its time evolution is governed by the $P$-matrix,
\begin{equation}\label{21jul2}
   \partial_t
    \begin{pmatrix}
    \psi_\pm(x,t;z) \\
    \partial_x \psi_\pm(x,t;z)
    \end{pmatrix}
    =P(x,t;z)
        \begin{pmatrix}
    \psi_\pm(x,t;z) \\
    \partial_x \psi_\pm(x,t;z)
    \end{pmatrix}.
\end{equation}
Note that we are now making the $z$ dependence explicit in our notation.

  In particular
  the Weyl $m$ functions
  are now defined uniquely (regardless of the choice of normalization of $\psi_\pm(x,0,z)$) as the logarithmic derivatives of the Weyl solutions,
  \[
  m_{\pm}(x,t;z)=\pm\frac{\partial_x\psi_\pm(x,t;z)}{\psi_\pm(x,t;z)}
  \]
  and \eqref{21jul2} allows us to determine their time evolution, as well as that of the Weyl $M$-matrix,
  \[
  M=\begin{pmatrix}
     \frac{-1}{m_{-}+m_{+}}& \frac12 \frac{m_{-}-m_{+}}{m_{-}+m_{+}} \\
    \frac12 \frac{m_{-}-m_{+}}{m_{-}+m_{+}} & \frac{m_{-}m_+}{m_-+m_{+}}
  \end{pmatrix}
  \]

The following lemma follows an argument described in \cite{R08} for the KdV equation; we repeat the argument for
the rest of the hierarchy for the sake of completeness.

  \begin{lem}
    \label{Mmatrix}
Let $q$ be a classical solution of the Cauchy problem \eqref{eq:kdvn}, \eqref{eq:initial} obeying the boundedness condition \eqref{21jul3}.  Then,
  for each $x\in \R$ and $z\in \C\setminus \bbR$, the Weyl $M$-matrix is differentiable in $t \in [0,T]$ and obeys
  \begin{equation}\label{22jul1}
\partial_tM=PM+MP^\top.
  \end{equation}
  \end{lem}

  \begin{proof}
    For fixed $z\in \C_+$, since the time evolution of $\psi_\pm$ and $\partial_x \psi_\pm$ is given by \eqref{21jul2}, the time evolution of $m_\pm$ is computed as
    \begin{align*}
    \partial_t m_\pm
    =\pm\frac{\partial_t \partial_x \psi_\pm}{\psi_\pm}-m_\pm\frac{\partial_t\psi_\pm}{\psi_\pm}
    =\pm P_{21}+(P_{22}-P_{11})m_\pm\mp P_{12}m_{\pm}^{2}.
    \end{align*}
    Then, the time derivatives of
    $m_1=\frac{-1}{m_{+}+m_{-}}$, $m_2=\frac{m_{-}m_+}{m_++m_{-}}$,
  $m_3=\frac12 \frac{m_{-}-m_{+}}{m_{+}+m_{-}}$ follow by mere calculations,
    \begin{align*}
      \partial_tm_{1} & =2P_{11}m_1+2P_{12}m_{3},\\
      \partial_tm_{2} & =-2P_{11}m_{2}+2P_{21}m_3,\\
      \partial_tm_{3} & =P_{21}m_1+P_{12}m_2.
    \end{align*}
    So by direct computation, we have $\partial_t M=PM+MP^\top$.
\end{proof}

\begin{cor}
\label{cor:timeinvariance}
Let $q$ be a classical solution of the Cauchy problem \eqref{eq:kdvn}, \eqref{eq:initial}
obeying the boundedness condition \eqref{21jul3}.
Then the spectrum $S = \sigma(H_{q(\cdot,t)})$
is independent of $t\in [0,T]$ and \eqref{22jul1} holds for all $z\in \bbC \setminus S$.
\end{cor}

\begin{proof}
We use a characterization of the spectrum of a Schr\"odinger operator
in terms of its Weyl $M$-matrix:
the spectrum of $H_{q(\cdot,t)}$ is the complement of the maximal set $\Omega \subset \C$ such that $M(0,t;z)$ has an analytic extension as a matrix-valued function on
$\Omega$ with the ``symmetry"
\begin{equation}\label{22jul2}
M(0,t;z) = M(0,t;\bar z)^*.
\end{equation}
We will show that the existence of such an analytic extension on some set is an invariant for the time evolution \eqref{22jul1}.

   Since $M(0,0;z)$ has an analytic extension with the property \eqref{22jul2} on $\bbC \setminus \sigma(H_{V})$, and the matrix $P$ is bounded for $t\in [0,T]$ and analytic in
    $z\in \C$, solving \eqref{22jul1} as an initial value problem starting from $t=0$ shows that $M(0,t;z)$ is also analytic on $\bbC \setminus \sigma(H_V)$. Moreover, since $P(\bar z) = \overline{P(z)}$, \eqref{22jul1} implies
    \[
    \partial_t M(0,t;\bar z)^* = P(0,t;z) M(0,t;\bar z)^* + M(0,t;\bar z)^* P(0,t;z)^\top.
    \]
In words, $M(0,t;\bar z)^*$ obeys the same time evolution as $M(0,t;z)$. Since $M(0,0;\bar z)^* = M(0,0;z)$, it follows that $M(0,t;\bar z)^* = M(0,t;z)$ for all $t$.

Since we have proved that $M(0,t;z)$ has an analytic extension with the property \eqref{22jul2} on $\bbC \setminus \sigma(H_{V})$, it follows that $\sigma(H_{q(\cdot,t)}) \subset \sigma(H_V)$ for all $t\in [0,T]$. Analogously, solving \eqref{22jul1} backward in time shows that $\sigma(H_V) \subset \sigma(H_{q(\cdot,t)})$ and completes the proof.
  \end{proof}

As Herglotz functions, $m_\pm$ have nontangential boundary values $m_\pm(\lambda+i0)$
 for Lebesgue-a.e. $\lambda \in \bbR$.
 Moreover, for a nontrivial Herglotz function, the boundary values are nonzero almost everywhere. In particular, for Lebesgue a.e. $\lambda$, the boundary values $m_\pm(\lambda+i0)$ exist and $(m_- + m_+)(\lambda+i0) \neq 0$.

This justifies the definition, for a.e. $\lambda \in \sigma(H_W)$, of the left and right reflection coefficient for a Schr\"{o}dinger operator $H_W$
defined as in \cite{R08} by
  \begin{align}
    R_{\pm}(x;\lambda):=\left( -\frac{m_{\mp}+\overline{m_{\pm}}}{m_{-}+m_{+}}\right) (x;\lambda+i0).
  \end{align}
If this quantity vanishes for Lebesgue-a.e. $\lambda \in \sigma(H_W)$, then
  \eqref{eq:reflectionlessness} holds and $H_{W}$ is reflectionless.
  Invariance of the reflectionless property under the translation flow was already proven in
  \cite[Corollary 3]{R08}.
  We note the generalization of the argument of \cite[Theorem 2]{R08} to the KdV-$n$ flow.
  In light of the above corollary, we may define the reflection coefficient of $H_{q(\cdot,t)}$
  as
  \begin{align}
    R_{\pm}(x,t;\lambda):=\left( -\frac{m_{\mp}+\overline{m_{\pm}}}{m_{-}+m_{+}}\right) (x,t;\lambda+i0)
  \end{align}
  for almost very $\lambda\in \sigma(H_V)$. We use Lemma~\ref{Mmatrix} to derive an ODE for
  the reflection coefficient.

  \begin{prop}
    Suppose
     $q$ obeys \eqref{eq:kdvn}, \eqref{eq:initial}, \eqref{21jul3}
    for some $T>0$, then
    \begin{equation}\label{25jul0}
       R_{\pm}(x,t;\lambda) =R_\pm(x,0;\lambda)\exp\left\{ 2i\int_0^t\Im(m_\pm(x,s;\lambda+i0))P_{12}(x,s;\lambda)\dd s\right\}
    \end{equation}
    for all $0\leq t\leq T$ and almost every $\lambda\in \sigma(H_V)$.
  \end{prop}

  \begin{proof}
    In addition to $z\in \bbC \setminus \sigma(H_V)$,
    the time evolution \eqref{22jul1} holds wherever
    the  Weyl $m$-functions have nontangential imits. Again, by basic properties of
    Herglotz functions,
    the nontangential limit, denoted $m_{\pm}(x,0,\lambda+i0)$ exists for
    Lebesgue almost
    every $\lambda\in \R$. So,
    $m_{\pm}(x,0;\lambda+i0)$ exists for Lebesgue almost every $\lambda\in \sigma(H_V)$
    and the evolution in \eqref{Mmatrix} may be extended to these $\lambda$.

    Thus, using \ref{Mmatrix} and the symmetry $P(\overline{z})=\overline{P(z)}$,
    we compute for almost every $\lambda\in \sigma(H_V)$,
    \begin{align*}
      \partial_tR_{\pm }(x,t;\lambda)&=2iP_{12}(x,t;\lambda) R_\pm(x,t,\lambda) \Im m_\pm(x,t;\lambda+i0)
    \end{align*}
    which implies \eqref{25jul0}.
  \end{proof}

In particular, since the reflectionless property corresponds to the case $R_\pm= 0$,  the proof of Prop.~\ref{reflectionlessness} is immediate.

Using the results above, we will prove Cor.~\ref{periodicity}. Our proof relies
on the following facts established in \cite{MO75,SY2}. By \cite{MO75},
if $S$ is the spectrum of a periodic operator with period $\ell$, we may associate to
any point $\{ \mu_j,\sigma_j\}_{j\in J}\in \cD(S)$ a periodic operator $H_Q$ with the same
Dirichlet data and period $\ell$.
Since $H_Q$, being periodic, is reflectionless, the bijective correspondence
between $\cR(S)$ and $\cD(S)$ established in \cite{SY2} proves
that $H_Q$ is the only
reflectionless operator corresponding to $\{ \mu_j,\sigma_j\}_{j\in J}$.

\begin{proof}[Proof of Cor.~\ref{periodicity}]
Let $S=\sigma(V)$ and fix $t\in [0,T]$.
By \eqref{cor:timeinvariance}, $\sigma(q(\cdot, t))=S$.
Associate to $q(\cdot,t)$ a point
$\{\mu_j(\cdot,t),\sigma_j(\cdot,t) \}\in \cD(S)$.
By \cite{MO75}, there exists an operator
$H_{Q(\cdot,t)}$ with Dirichlet data $\{\mu_j(\cdot,t),\sigma_j(\cdot,t) \}\in \cD(S)$
and such that $Q(\cdot+\ell, t)=Q(\cdot, t)$.
By Prop.~\ref{reflectionlessness},
$q(\cdot,t)\in \cR(S)$. Furthermore, by the bijection established in
\cite{SY2}, and since $Q(\cdot,t)\in \cR(S)$, we have the equality
$q(\cdot, t)=Q(\cdot,t)$.
Thus, $q(\cdot+\ell,t)=q(\cdot,t)$ for all $t$, so $q$ is a spatially periodic solution of KdV-$n$, which implies $q(\cdot,t) = \cS_t(V)$.
\end{proof}

\section{Diagonal Green's function and trace formulas}

In order to turn our attention to Dirichlet eigenvalues, we have to review some facts about the diagonal Green's function. Consider the Schr\"odinger operator $H_{q(\cdot,t)}$; denote its Weyl solutions by $\psi_\pm(x,t;z)$ and $m$-functions by $m_\pm(x,t;z)$. Its diagonal Green's function can be expressed in terms of Weyl solutions as
\[
G(x,x,t;z) = - \frac 1{m_-(x,t;z) + m_+(x,t;z)}.
\]
It is well known (see \cite[Theorem 4.5]{S99} and \cite{GHSZ})
  that the Weyl $m$-functions for a potential $W$ which is  $2n+1$ times differentiable
  in $x$
  have an asymptotic expansion of the form
  \begin{align*}
    m_\pm(x,t;-k^2) =-k + \sum_{j=1}^{2n+1} (\pm 1)^j c_{j+1}(x,t) k^{-j}+o\left( k^{-2n-1} \right)
  \end{align*}
  as $k\to \infty$ nontangentially, $\epsilon<\arg(k)<\frac{\pi}{2}-\epsilon$ for some $\epsilon>0$.
  The coefficients $c_j$ can be computed using the Ricatti equation for $m_\pm$; likewise, the diagonal Green's function is found to have a similar expansion and its coefficients can be computed \cite[Equation D.22]{GH03} using the identity
  \begin{align*}
    -2G \partial_x^2 G +(\partial_x G)^2+4(q-z)G^2=1
  \end{align*}
which results in the asymptotic expansion
\begin{align}
\label{originalexpansion}
G(x,x,t;-k^2) = \frac 1{2} \sum_{\ell=0}^{n+1} \hat f_\ell k^{-2\ell-1} + o(k^{-2n-3})
\end{align}
valid as $k\to \infty$ nontangentially, $\epsilon<\arg(k)<\frac{\pi}{2}-\epsilon$ for some $\epsilon>0$. Here $\hat f_\ell$ are precisely the quantities defined recursively in the introduction.

We will  assume from now on that $S$ obeys the moment condition \eqref{eq:Craig5}. We consider the scalar fields $R_m$ on $\cD(S)$ for $m\le n$.
Our goal will be to prove that these are continuous scalar fields on $\cD(S)$ and that they correspond precisely to $\hat f_\ell$ for our solution. In other words, we will prove that our solution $q(x,t)$ obeys certain trace formulas.

  We will use the following estimate from \cite[Lemma 4.1]{BDGL18}: if $S$ obeys the condition \eqref{eq:Craig5},
  then $Q_k$ is a continuous scalar field on $\cD(S)$ for any $k \le n+1$. It is bounded by
  \begin{align}
    \label{ineq:BDGLestimate}
    Q_{k}\leq
      |\underline{E}|^k+3D_{k}\sum_{j\in J}(1+\eta_{j,0}^{k-1})\gamma_j
  \end{align}
where the constant $D_k$ depends only on the index $k$, $\underline{E}$ and
  $\sup_{j\in J}\gamma_j$.

It follows that $R_m$ defined by \eqref{eq:Rmdefn} is also a continuous scalar field on $\cD(S)$
for $1\leq m\leq n$.
This definition will be motivated
by the fact that for solutions $q(x,t)$
of the Cauchy problem, the trace
formulas will be of the form $\hat{f}_m= R_m$,
 in the sense given in Lemma~\ref{lem:compfn} below.  We begin with an elementary inequality (which will be repeatedly useful) and then proceed to derive some estimates on these scalar fields.

  \begin{lem}
      \label{lem:estimate1} \marginpar{}
      For $\sigma(H_V)$ obeying
      \eqref{eq:Craig5},
      there are constants $M_1$, $M_3$ such that for $1\leq m\leq n$,
      \[
      |R_{m }|\leq M_1
      \]
      and for any $k\in J$,
      \begin{equation}\label{7aug1}
      \left|\frac{\partial R_{m }}{\partial \varphi_k}\right|\leq
      M_3(1+(\eta_{k,0}+\gamma_k)^m)\gamma_k.
      \end{equation}
      \end{lem}

      \begin{proof}
      Note that $R_m$ is a polynomial in the $Q_k$.
      For the sake of clarity, we denote for $1\leq m\leq n$,
      \begin{align*}
        R_{m }=P_{m}(Q_1,\dots,Q_{m})
      \end{align*}
      In the above, the $P_{m}$ are polynomials of degree $m$ in the first $m$
      scalar fields defined
      in \eqref{eq:higherordertrace}. Due to the moment condition \eqref{eq:Craig5},
      the $Q_k$ for $1\leq k\leq n+1$ are continuous, and thus, so are the $R_m$.
      Also by condition \eqref{eq:Craig5}, $\mathcal{D}(S)$ with the metric
      \eqref{eq:metric} is compact. Thus, we have a uniform bound on each $R_m$ for
      $1\leq m\leq n$. Taking the maximum of these yields $M_1$.

      Noting that
      \begin{align*}
        n\geq \max_{1\leq k,j\leq n}\deg\left(\frac{\partial P_k}{\partial Q_j} \right),
      \end{align*}
      we write
      \begin{align*}
        \frac{\partial P_m}{\partial Q_j}=\sum_{|\alpha|\leq m}A_\alpha\prod_{k=1}^{m}Q_k^{\alpha_k}
      \end{align*}
      We then take $A$ to be the maximum of all the coefficients of the
      $\frac{\partial P_m}{\partial Q_j}$ for $1\leq m,j \leq n$ and
      using the notation in estimate \eqref{ineq:BDGLestimate},
      define $C=\max_{1\leq k\leq n}D_{k}$ for convenience.
      We also note the elementary fact
      $1+x^k\leq 2(1+x^{n})$ for $0\leq k\leq n$, and $x> 0$.
      Then for any $m$ and $j$ we have the estimate, uniform in $m$ and $j$,
      \begin{align*}
        \left| \frac{\partial P_m}{\partial Q_j} \right|&\leq
        A\sum_{|\alpha|\leq m}\prod_{k=1}^{m}(|\underline{E}|^k+3D_k\sum_{i\in J}(1+\eta_{i,0}^{k-1})\gamma_i)^{\alpha_k}\\
        &\leq A\sum_{|\alpha|\leq m}\prod_{k=1}^{m}3^{\alpha_k}(1+|\underline{E}|^n)^{\alpha_k}(1+D_k\sum_{i\in J}(1+\eta_{i,0}^{k-1})\gamma_i)^{\alpha_k}\\
        &\leq 3^n(1+|\underline{E}|^n)^{n}\sum_{|\alpha|\leq m}\prod_{k=1}^{m}(1+2C\sum_{i\in J}(1+\eta_{i,0}^{n-1})\gamma_i)^{\alpha_k}\\
        &\leq 3^n(1+|\underline{E}|^n)^{n}\left(1+2C\sum_{i\in J}(1+\eta_{i,0}^{n-1})\gamma_i\right)^{n}\left({2n\choose n}-1 \right)\\
        &=:M_2
      \end{align*}
      Thus, for $1\leq m \leq n$ we have
      \begin{align*}
        \left|\frac{\partial R_{m }}{\partial \varphi_k}\right|
        &=
        \left|
        \frac{\partial P_{m}}{\partial Q_1}\frac{\partial Q_1}{\partial \varphi_k}+\cdots+
        \frac{\partial P_{m}}{\partial Q_m}\frac{\partial Q_m}{\partial \varphi_k} \right|\\
        & \leq
        \left|\frac{\partial P_{m}}{\partial Q_1}\right|\left|\frac{\partial Q_1}{\partial \varphi_k}\right|+\cdots+
        \left|\frac{\partial P_{m}}{\partial Q_m}\right|\left|\frac{\partial Q_m}{\partial \varphi_k}\right|
        \\
        &\leq M_2\gamma_k\left( 1+2(|\underline{E}|+\eta_{k,0}+\gamma_{k})+\cdots+m(|\underline{E}|+\eta_{k,0}+\gamma_{k})^{m-1}\right)\\
        &\leq \frac{m(m+1)}{2}(1+(|\underline{E}|+\eta_{k,0}+\gamma_k)^m)M_2\gamma_k.
      \end{align*}
By using the power mean inequality
\begin{equation}\label{7aug2}
1+(|\underline{E}|+\eta_{k,0}+\gamma_k)^m \le 2^{m-1}(1+|\underline{E}|^m)(1+(\eta_{k,0}+\gamma_k)^m),
\end{equation}
this implies \eqref{7aug1} with $M_3:=\frac{n(n+1)}{2}2^{n}(1+|\underline{E}|^n)M_2$.
\end{proof}

For a reflectionless solution $q$, let us define for each $(x,t) \in \bbR \times [0,T]$  the set of Dirichlet data $\varphi(x,t) \in \cD(S)$. Begin by defining
  \begin{align}
    \label{eq:Dirichleteig2}
    \mu_j(x,t) =\begin{cases}
    z\in (E_j^-,E_j^+), & G(x,x,t;z)=0\\
    E_j^+, & G(x,x,t;z)<0\;\;\forall z\in (E_j^-,E_j^+)\\
    E_j^-, & G(x,x,t;z)>0\;\;\forall z\in (E_j^-,E_j^+)
    \end{cases}
  \end{align}
If $\mu_j(x,t) \in (E_j^-, E_j^+)$, we also define $\sigma_j = - \partial_x G(x,x,t;z)\vert_{z=\mu_j(x,t)} \in \{\pm 1\}$. Then $\varphi_j(x,t)$ are defined by \eqref{eq:coordinates1}, \eqref{eq:coordinates2}.

We will now prove that the values of differential polynomials $\hat f_\ell$ evaluated at some $(x,t)$ are precisely given by the scalar fields $R_\ell$ evaluated at the Dirichlet data $\varphi(x,t)$.

\begin{lem}
  \label{lem:compfn}
Let $q$ be a classical solution of the Cauchy problem \eqref{eq:kdvn}, \eqref{eq:initial}
obeying the boundedness condition
\eqref{21jul3}. Let $V \in \cR(S)$
where $S$ obeys the moment condition
\eqref{eq:Craig5}. For all $1 \le m \le n$ and all $(x,t) \in \bbR \times [0,T]$,
\[
\hat f_m (x,t) = R_m(\varphi(x,t)).
\]
\end{lem}

These are higher-order trace formulas; the case $m=1$ is precisely \eqref{21jul4}. We also emphasize that the above lemma represents $\hat{f}_n$ as a polynomial in the moments $Q_k$.

  \begin{proof}
      As noted \cite[Equation 2.7]{EVY18}, the diagonal Green's function has an
      exponential Herglotz
      representation. The representation takes the form
      \begin{align}
        \label{eq:HerglotzGreen}
        G(x,x,t;z)&=\frac{1}{2\sqrt{z-\underline{E}}}e^{-\int_{\underline{E}}^\infty \frac{f(\xi)}{\xi-z}\dd \xi}
      \end{align}
      again with the positive real axis is chosen as the branch cut for the square root,
      and where
      \begin{align*}
        f(\xi)&=\frac{1}{2}-\frac1\pi\arg G(x,x,t;\xi+i0).
      \end{align*}
      The reflectionless property of the potential yields
      \begin{align*}
        f(\xi)&=\begin{cases}
            0,&\xi\in S\\
            -\frac12,&\xi\in (E_{j}^-,\mu_j)\\
            \frac12,&\xi\in (\mu_j,E_{j}^+)
        \end{cases}
      \end{align*}
By the moment condition \eqref{eq:Craig5}, for $k=1,\dots,n+1$,
\[
2k \int_{\underline E}^\infty \xi^{k-1} f(\xi)\,d\xi = Q_k - \underline{E}^k
\]
Assuming $\epsilon < \arg(z-\underline E) < 2 \pi - \epsilon$ for some $\epsilon > 0$  and $\xi\geq \underline{E}$, we can estimate
\[
\left\lvert \frac 1{\xi-z} + \sum_{k=1}^n \frac{\xi^{k-1}}{z^k} \right\rvert = \left \lvert \frac{ \frac{\xi^n}{z^n}}{\xi-z} \right\rvert \le \left( 1+ \frac 1{\sin \epsilon} \right) \left\lvert \frac{\xi^{n}}{z^{n} (z-\underline E)} \right\rvert.
\]
Multiplying by $f(\xi)$ and integrating implies that
      \begin{align*}
      &\left| \int_{\underline{E}}^\infty\frac{1}{\xi-z}f(\xi)\dd \xi + \sum_{k=1}^n \frac{Q_k-\underline{E}^k}{2kz^k} \right|
      \le  \left( 1+ \frac 1{\sin \epsilon} \right) \frac{\lvert Q_{n+1} - \underline{E}^{n+1}\rvert}{|z|^{n} \lvert z-\underline{E} \rvert},
      \end{align*}
which in turn yields the asymptotic expansion
      \begin{equation}\label{3aug1}
        G(x,x,t;z) = \frac{1}{2\sqrt{z-\underline{E}}}e^{\sum_{k=1}^{n}\frac{Q_{k}-\underline{E}^k}{2kz^{k}}+o(z^{-n})}
      \end{equation}
      as $z \to \infty$, $\arg (z-\underline{E}) \in (\epsilon,2\pi-\epsilon)$.

In order to compare coefficients to the expansion in \eqref{originalexpansion}, we write
\[
\frac{1}{2\sqrt{z-\underline{E}}} = \frac{1}{2\sqrt{z}}e^{- \frac 12 \log(1-\frac{\underline{E}}{z})}=\frac{1}{2\sqrt{z}} e^{\sum_{k=1}^n\frac{\underline{E}^k}{2kz^k}+o(z^{-n})}
\]
and combine with \eqref{3aug1} to find
      \begin{align*}
        G(x,x,t;z)&=\frac{1}{2\sqrt{z}}e^{\sum_{k=1}^{n}\frac{Q_{k}}{2kz^{k}}+o(z^{-n})}\\
        &=
        \frac{1}{2}\sum_{m=0}^n \frac{1}{z^{m+\frac12}}\sum_{\substack{\alpha\in \N_0^n\\ \sum_{k=1}^m k\alpha_k=m}}
        \prod_{k=1}^n\frac{Q_k^{\alpha_k}}{\alpha_k!(2k)^{\alpha_k}}
        +o(z^{-n-\frac 12})
      \end{align*}
Comparing coefficients to the expansion in \eqref{originalexpansion} completes the proof.
      \end{proof}

The equation \eqref{eq:HerglotzGreen} also yields a product formula for the diagonal
Green's function common in the literature:

\begin{align}
\label{eq:Greenproduct}
G(x,x,t;z)=\frac12\sqrt{\frac{1}{\underline{E}-z}\prod_{\ell\in J}\frac{(\mu_{\ell}(x,t)-z)^2}{(E_{\ell}^{-}-z)(E_{\ell}^+-z)} }.
\end{align}

\section{Time evolution of Dirichlet eigenvalues}

Now we will derive the time evolution for $\varphi_j$ in terms of a Lipschitz
  vector field on the torus $\mathcal{D}(S)$ equipped with the metric \eqref{eq:metric}.

  \begin{lem}
      \label{lem:varphievolution}
      Suppose $q(x,t)$ a solution to \eqref{eq:kdvn}, \eqref{eq:initial} satisfying
      \eqref{21jul3}. Let $(\varphi_j(x,t))_{j\in J}$
      be the corresponding Dirichlet data
      defined in \eqref{eq:Dirichleteig}, \eqref{eq:coordinates1} and \eqref{eq:coordinates2}.
      If $(x,t)$ is such
      that
      $\varphi_j(x,t)\notin \pi \Z$, then
      \begin{align}
        \partial_t\varphi_j(x,t)=\Xi_j(\varphi(x,t)).
      \end{align}
      \end{lem}

  \begin{proof}
    By the definitions \eqref{eq:coordinates1}, \eqref{eq:coordinates2} and \eqref{eq:Dirichleteig}, if $\varphi_j(x,t)\notin \pi \Z$, then
    \begin{align*}
    G(x,x,t;\mu_j(x,t))=0
    \end{align*}
so by the implicit function theorem,
    \begin{align*}
      \frac{\partial\mu_j}{\partial t}=-\frac{\partial G(x,x,t;z)/\partial t\vert_{z=\mu_j(x,t)}}{\partial G(x,x,t;z)/\partial z\vert_{z=\mu_j(x,t)}}.
    \end{align*}
Using \eqref{eq:coordinates1} and \eqref{eq:coordinates2} we have differentiability of $\varphi_j$ and the explicit expression
    \begin{align*}
      \frac{\partial\varphi_j}{\partial t}=
      \frac{\partial\mu_j/\partial t}{-\sigma_j(x,t)\sqrt{(E_j^+-\mu_j)(\mu_j-E_j^{-})}}.
    \end{align*}
    Noting that $G(x,x,t;z)=m_1$, the upper left entry of the Weyl $M$-matrix $M(x,t;z)$, we read off the time evolution for
    the diagonal Green's function from Lemma~\ref{Mmatrix},
    \begin{align*}
      \partial_tm_1(x,t,z)=2P_{12}m_3(x,t,z)+2P_{11}m_1(x,t,z).
    \end{align*}
    Then, plugging in $z=\mu_j(x,t)\in \C\setminus S$ in the above we see
    \begin{align*}
      \partial_tm_1(x,t,\mu_j(x,t))=-\sigma_j(x,t)\hat{F}_n(\mu_j(x,t)).
    \end{align*}
    Differentiating the product formula \eqref{eq:Greenproduct}  in $z$, we find
    \begin{align*}
      \partial G(x,x,t;z)/\partial z\vert_{z=\mu_j(x,t)}&=
      \frac12 \sqrt{\frac{1}{(\underline E - \mu_j)(E_j^{-}-\mu_j)(E_j^{+}-\mu_j)}
      \prod_{l\ne j}\frac{(\mu_l-\mu_j)^2}{(E_l^{-}-\mu_j)(E_l^{+}-\mu_j)}  }.
    \end{align*}
    So, in combination, we have
    \[
          \frac{\partial\varphi_j}{\partial t}=\left(\sum_{\ell=0}^nR_{n-\ell}(\mu_j(x,t))^\ell\right)\Psi_j. \qedhere
  \]
    \end{proof}

  We now note that by arguments identical to those found in \cite[Lemma 3.4, Lemma 3.5]{BDGL18} show that the above evolution equation for $\varphi_j$ holds also for
  $\varphi_j(x,t)\in \pi\Z$, i.e. at the gap edges.
  Indeed, Lemma 3.4 of \cite{BDGL18}
  relies only on continuity of $G(t,x,x,z)$ in $z$ and $t$
  to prove $\varphi_j(x,t)$ is continuous in $t$.
  Lemma 3.5 of \cite{BDGL18} is an intermediate result which relies only on
  the continuity of the vector field $\Xi_j$ in time.
  This hypothesis is clearly satisfied in our case once we note that
  each $R_n$ is a polynomial in the
  $Q_n$, which are continuous in the $\mu_j$ by \eqref{eq:Craig1}.
  In conclusion, we have the following result.

  \begin{prop}
      \label{lem:varphievolution2}
      Suppose $q(x,t)$ a solution to \eqref{eq:kdvn}, \eqref{eq:initial} satisfying
      \eqref{21jul3}. Let $(\varphi_j(x,t))_{j\in J}$
      be the corresponding Dirichlet data
      defined in \eqref{eq:Dirichleteig}, \eqref{eq:coordinates1} and \eqref{eq:coordinates2}.
      For any $(x,t)\in \R^2$ and $j\in J$,
      \begin{equation}\label{5aug1}
        \partial_t\varphi_j(x,t)=\Xi_j(\varphi(x,t)).
      \end{equation}
    \end{prop}

    Now that we know \eqref{5aug1} holds for all $x,t$, we can view this collection of equations indexed by $j$ as an autonomous system of ODEs on $\cD(S)$, generated by the vector field $\Xi$. We now establish sufficient conditions for this vector field to be Lipschitz on the torus: this will allow us to conclude uniqueness of solution of \eqref{5aug1}; as is already established, the solution to \eqref{eq:kdvn}, \eqref{eq:initial} is uniquely recovered from the Dirichlet data.
    \begin{lem}
        \label{lem:Lipschitz}
        If the set $S = \sigma(H_V)$ obeys
        \eqref{eq:Craig1}, \eqref{eq:Craig2}, \eqref{eq:Craig3}, \eqref{eq:Craig4},
        the vector field $\Xi$ with components defined by \eqref{eq:Xidef}
        is a Lipschitz vector field on $\mathbb{T}^J$.
        \end{lem}
    \begin{proof}
        Computing directly,
        \begin{align*}
        &\frac{\partial \Xi_j}{\partial \varphi_k}=
        \left( \sum_{\ell=0}^{n-1} \frac{\partial R_{n-\ell}}{\partial \varphi_k}\mu_j^\ell \right) \Psi_j
        +\left(\sum_{\ell=0}^n R_{n-\ell}\mu_j^\ell\right)\frac{\partial \Psi_j}{\partial \varphi_k},\;\;k\ne j\\
        &\frac{\partial \Xi_j}{\partial \varphi_j}=
        \left( \sum_{\ell=0}^{n-1} \frac{\partial R_{n-\ell}}{\partial \varphi_j}\mu_j^\ell +
         \sum_{\ell=1}^{n-1}\ell R_{n-\ell}\frac{\partial \mu_j}{\partial \varphi_j}\mu_j^{\ell -1}\right) \Psi_j+
        \left(\sum_{\ell=0}^n R_{n-\ell}\mu_j^\ell\right)\frac{\partial \Psi_j}{\partial \varphi_j}
        \end{align*}
        We will now show that $\left| \Psi_j\right|\leq 2C_j $; this follows from
        maximizing the smooth function
        $f(s)=\frac{(E_{\ell}^+-\mu_j)(E_{\ell}^{-}-\mu_j)}{(E_{\ell}^{-}+s\gamma_\ell-\mu_j)^2}$
        for $0\leq s\leq 1$. For $E_j^{+}<E_{\ell}^{-}$, we see that $f$ decreases
        and,
        \begin{align*}
        \frac{(E_\ell^{-}-\mu_j)(E_{\ell}^+-\mu_j)}{(\mu_j-\mu_{\ell})^2}&\leq
        \frac{E_{\ell}^+-\mu_j}{E_{\ell}^{-}-\mu_j}
        = 1+\frac{\gamma_\ell}{E_{\ell}^{-}-\mu_j}
        \leq
        1+\frac{\gamma_\ell}{\eta_{j,\ell}},
        \end{align*}
        and similarly for $E_j^{-}>E_{\ell}^{+}$, where $f$ increases.
        And, thus we have our bound on $\Psi_j$,
        \begin{align*}
        \Psi_j&=2\sqrt{(\mu_j-\underline{E})\prod_{\ell\ne j}\frac{(E_\ell^{-}-\mu_j)(E_{\ell}^+-\mu_j)}{(\mu_j-\mu_{\ell})^2}}\\
        &\leq
        2(\eta_{j,0}+\gamma_j)^{1/2}\prod_{\ell\ne j}\left(1+\frac{\gamma_\ell}{\eta_{j,\ell}}\right)^{1/2}\\
        &=2C_j.
        \end{align*}
        We will also need that $\left| \frac{\partial\Psi_j}{\partial \varphi_k}\right|\leq \frac{\gamma_k}{2\eta_{j,k}}C_j$
        when $k\ne j$. Note that
        $\left|\frac{\partial \mu_k}{\partial \varphi_k} \right|\leq \frac12 \gamma_k$
        by a differentiation of \eqref{eq:coordinates1}. Then we estimate,
        \begin{align*}
          \Psi_j&=2\sqrt{(\mu_j-\underline{E})\prod_{\ell\ne j,k}\frac{(E_\ell^{-}-\mu_j)(E_{\ell}^+-\mu_j)}{(\mu_j-\mu_{\ell})^2}}\cdot
          \sqrt{\frac{(E_k^{-}-\mu_j)(E_{k}^+-\mu_j)}{(\mu_j-\mu_{k})^2}}\\
          &\implies \left|\frac{\partial \Psi_j}{\partial \varphi_k}\right|=|\Psi_j|\frac{1}{|\mu_j-\mu_k|}\left|\frac{\partial \mu_k}{\partial \varphi_k}\right|
          \leq  C_j\frac{\gamma_k}{\eta_{j,k}}.
        \end{align*}
        Thus, for $k\ne j$,
        \begin{align}
          \label{ineq:partialk}
          \left| \frac{\partial \Xi_j}{\partial \varphi_k}\right|&\leq
        2C_j\sum_{\ell=0}^{n-1} \left|\frac{\partial R_{n-\ell}}{\partial \varphi_k}\right|(\lvert\underline{E}\rvert+\eta_{j,0}+\gamma_j)^\ell
        +\frac{\gamma_k}{\eta_{j,k}}C_j\sum_{\ell=0}^n |R_{n-\ell}|(\lvert\underline{E}\rvert+\eta_{j,0}+\gamma_j)^\ell.
        \end{align}
The second sum is bounded using the estimates in Lemma~\ref{lem:estimate1} as
        \begin{align*}
          \frac{\gamma_k}{\eta_{j,k}}C_j\sum_{\ell=0}^n |R_{n-\ell}|(\lvert\underline{E}\rvert+\eta_{j,0}+\gamma_j)^\ell
          &\leq\frac{\gamma_k}{\eta_{j,k}}C_jM_1\sum_{\ell=0}^n(\lvert\underline{E}\rvert+\eta_{j,0}+\gamma_j)^\ell\\
          &\leq \frac{\gamma_k}{\eta_{j,k}}(n+1)C_j M_1(1+(\lvert\underline{E}\rvert+\eta_{j,0}+\gamma_j)^n)\\
          &\leq \tilde{M}_1(1+(\eta_{j,0}+\gamma_j)^n)\frac{\gamma_k}{\eta_{j,k}}C_j
        \end{align*}
        where $\tilde{M}_1:=2^{n-1}(n+1)M_1(1+|\underline{E}|^n)$ and the last step uses \eqref{7aug2}.

        For the first sum in \eqref{ineq:partialk},
         we again use
        Lemma~\ref{lem:estimate1} and \eqref{7aug2}, as well as the inequality
        $1+x^k\leq 2(1+x^{n})$ for $0\leq k\leq n$, and $x> 0$.
        Indeed, we have
        \begin{align*}
          2C_j\sum_{\ell=0}^{n-1} \left|\frac{\partial R_{n-\ell}}{\partial \varphi_k}\right|(|\underline{E}|+\eta_{j,0}+\gamma_j)^\ell
          &\leq
          4M_3\gamma_kC_j(1+(\eta_{k,0}+\gamma_k)^n)\sum_{\ell=0}^{n-1} (|\underline{E}|+\eta_{j,0}+\gamma_j)^\ell \\
          &\leq 4M_3(1+|\underline{E}|^n)\gamma_kC_j(1+(\eta_{k,0}+\gamma_k)^n)\sum_{\ell=0}^{n-1} 2^{\ell}(1+(\eta_{j,0}+\gamma_j)^\ell) \\
          &\leq \tilde{M}_3\gamma_kC_j(1+(\eta_{k,0}+\gamma_k)^n)(1+(\eta_{j,0}+\gamma_j)^n)
        \end{align*}
        where $\tilde{M}_3:=2^{n+2}nM_3(1+|\underline{E}|^n)$.
        Combining our estimates with \eqref{ineq:partialk}, and applying
        the power mean inequality, we see for $k\ne j$,
        \begin{align*}
          \left|\frac{\partial \Xi_j}{\partial \varphi_k}\right|&\leq
          \tilde{M}_1C_j\frac{\gamma_k}{\eta_{j,k}} (1+(\eta_{j,0}+\gamma_j)^n)+
          \tilde{M}_3C_j\gamma_k(1+(\eta_{k,0}+\gamma_k)^n)(1+(\eta_{j,0}+\gamma_j)^n)\\
          &\leq \max\{ \tilde{M}_1,\tilde{M}_3\}C_j(1+(\eta_{j,0}+\gamma_j)^n)
          \left(\frac{\gamma_k}{\eta_{j,k}}
          +\gamma_k(1+(\eta_{k,0}+\gamma_k)^n) \right)\\
          &\leq 2^{n-1}\max\{ \tilde{M}_1,\tilde{M}_3\}C_j(1+\eta_{j,0}^n+\gamma_j^n)
          \left(\frac{\gamma_{k}}{\eta_{j,k}}
          +\gamma_{k}(1+2^{n-1}(\eta_{k,0}^n+\gamma_{k}^n)) \right)\\
          &\leq 2^{2n-2}\max\{ \tilde{M}_1,\tilde{M}_3\}(1+\sup_{m\in J}\gamma_m^n)C_j(1+\eta_{j,0}^n)
          \left(\frac{\gamma_{k}}{\eta_{j,k}}
          +(1+\sup_{m\in J}\gamma_m^n)\gamma_{k}(1+\eta_{k,0}^n) \right)\\
          &\leq 2^{2n-2}\max\{ \tilde{M}_1,\tilde{M}_3\}(1+\sup_{m\in J}\gamma_m^n)^2C_j(1+\eta_{j,0}^n)
          \left(\frac{\gamma_{k}}{\eta_{j,k}}
          +\gamma_{k}(1+\eta_{k,0}^n) \right)\\
          &\leq \tilde{M}C_j(1+\eta_{j,0}^n)
          \left(\frac{\gamma_{k}}{\eta_{j,k}}
          +\gamma_{k}(1+\eta_{k,0}^n) \right).
        \end{align*}
        Where we have defined $\tilde{M}:=2^{2n-2}\max\{ \tilde{M}_1,\tilde{M}_3\}(1+\sup_{m\in J}\gamma_m^n)^2$ in
        the above for
        convenience. Note that $\tilde{M}$ is finite by \eqref{eq:Craig1}, which implies $\sum_{m\in J}\gamma_m<\infty$ so that
        of course $\sup_{m\in J}\gamma_m^n<\infty$.

        For $k=j$, in addition to the estimates above, we will need to show the
        following bound on $\frac{\partial \Psi_j}{\partial \varphi_j}$,
        \[
        \left|\frac{\partial \Psi_j}{\partial \varphi_j} \right|\leq
        \frac{C_j\gamma_j}{2}\left( \frac{1}{\eta_{j,0}}+\sum_{\ell\ne j}\frac{\gamma_\ell}{\eta_{j,\ell}(\eta_{j,\ell}+\gamma_\ell)} \right)
        \]
        This follows by first, another maximization argument, which finds
        \begin{align*}
          \left|\frac{2}{\mu_\ell-\mu_j}-\frac{1}{E_{\ell}^+-\mu_j}-\frac{1}{E_{\ell}^{-}-\mu_j} \right|\leq
          \frac{\gamma_\ell}{\eta_{j,\ell}(\gamma_\ell+\eta_{j,\ell})}.
        \end{align*}
        The above justifies the termwise differentiation of the
        sum over the index $J$ and shows
        \begin{align*}
          \left|\frac{\partial \Psi_j}{\partial \varphi_j}\right|&=
          \left|\frac{\Psi_j}{2}\frac{\partial \mu_j}{\partial\varphi_j}\left(\frac{1}{\mu_j-\underline{E}}+
          \sum_{\ell\ne j} \left( \frac{2}{\mu_\ell-\mu_j}-\frac{1}{E_{\ell}^+-\mu_j}-
          \frac{1}{E_{\ell}^--\mu_j} \right) \right)\right|\\
          &\leq
          \frac{C_j\gamma_j}{2}\left( \frac{1}{\eta_{j,0}}+\sum_{\ell\ne j}
          \frac{\gamma_\ell}{\eta_{j,\ell}(\gamma_\ell+\eta_{j,\ell})}  \right).
        \end{align*}
        For the first term in $\frac{\partial \Xi_j}{\partial \varphi_j}$, we use \eqref{7aug2} to find
        \begin{align*}
        \sum_{\ell=0}^{n-1} \left|\frac{\partial R_{n-\ell}}{\partial \varphi_j}\right||\mu_j|^\ell
        &\leq M_3\gamma_j\sum_{\ell=0}^{n-1}(1+(\eta_{j,0}+\gamma_j)^{n-\ell})(|\underline{E}|+\eta_{j,0}+\gamma_j)^\ell\\
        &\leq 2^{n-1}M_3(1+|\underline{E}|^n)\gamma_j
        \sum_{\ell=0}^{n-1}(1+(\eta_{j,0}+\gamma_j)^{n-\ell})(1+(\eta_{j,0}+\gamma_j)^{\ell})\\
        &\leq 2^{n+1}nM_3 (1+|\underline{E}|^n)\gamma_j(1+(\eta_{j,0}+\gamma_j)^{n})
        \\
        &= \frac12\tilde{M}_3\gamma_j(1+(\eta_{j,0}+\gamma_j)^{n})
        \end{align*}
        since $(1+(\eta_{j,0}+\gamma_j)^{n-\ell})(1+(\eta_{j,0}+\gamma_j)^{\ell})\leq 4(1+(\eta_{j,0}+\gamma_j)^{n})$.
        Similar methods bound the remaining two terms:
        \begin{align*}
        \sum_{\ell=1}^{n-1}\ell|R_{n-\ell}|\left|\frac{\partial \mu_j}{\partial \varphi_j}\right||\mu_j|^{\ell -1}
        &\leq
        \frac{(n-1)M_1}{2}\gamma_j\sum_{\ell=1}^{n-1}(|\underline{E}|+\eta_{j,0}+\gamma_j)^{\ell -1} \\
        &\leq \frac{(n-1)M_1}{2}(1+|\underline{E}|^{n})\gamma_j\sum_{\ell=1}^{n-1}2^{\ell-1}(1+(\eta_{j,0}+\gamma_j)^{\ell -1}) \\
        &\leq 2^{n-1}\frac{(n-1)^2M_1}{2}(1+|\underline{E}|^{n})\gamma_j(1+(\eta_{j,0}+\gamma_j)^{n}) \\
        &\leq \frac{(n-1)}{2}\tilde{M}_1\gamma_j(1+(\eta_{j,0}+\gamma_j)^n)
        \end{align*}
        and
        \begin{align*}
        \sum_{\ell=0}^n |R_{n-\ell}||\mu_j|^\ell&\leq
        M_1\sum_{\ell=0}^n (|\underline{E}|+\eta_{j,0}+\gamma_j)^\ell\\
        &\leq M_1(1+|\underline{E}|^n)\sum_{\ell=0}^n 2^{\ell}(1+(\eta_{j,0}+\gamma_j)^\ell)\\
        &\leq 2^{n+1}M_1(n+1)(1+|\underline{E}|^n)(1+(\eta_{j,0}+\gamma_j)^n)\\
        &=4\tilde{M}_1(1+(\eta_{j,0}+\gamma_j)^n).
        \end{align*}

        Thus, we have
        \begin{align*}
          \left|\frac{\partial \Xi_j}{\partial \varphi_j}\right|&\leq
          2C_j\left( \sum_{\ell=0}^{n-1} \left|\frac{\partial R_{n-\ell}}{\partial \varphi_j}\right||\mu_j|^\ell +
           \sum_{\ell=1}^{n-1}\ell|R_{n-\ell}|\left|\frac{\partial \mu_j}{\partial \varphi_j}\right||\mu_j|^{\ell -1}\right) +
          \left(\sum_{\ell=0}^n |R_{n-\ell}||\mu_j|^\ell\right)\left|\frac{\partial \Psi_j}{\partial \varphi_j}\right|\\
          &\leq
          \tilde{M}_3C_j\gamma_j\left(1+(\eta_{j,0}+\gamma_j)^{n}\right)\\
          &\qquad+(n-1)\tilde{M}_1C_j\gamma_j(1+(\eta_{j,0}+\gamma_j)^n)\\
          &\qquad+2\tilde{M}_1C_j\gamma_j\left( \frac{1}{\eta_{j,0}}+\sum_{k\ne j}\frac{\gamma_k}{\eta_{j,k}(\eta_{j,k}+\gamma_k)} \right)
          (1+(\eta_{j,0}+\gamma_j)^n)\\
          &\leq \tilde{M}_32^{n-1}(\sup_{m\in J}\gamma_m^n+1)C_j\gamma_j\left(1+\eta_{j,0}^n\right)\\
          &\qquad +(n-1)\tilde{M}_12^{n-1}(\sup_{m\in J}\gamma_m^n+1)C_j\gamma_j(1+\eta_{j,0}^n)\\
          &\qquad+2\tilde{M}_12^{n-1}(\sup_{m\in J}\gamma_m^n+1)C_j\gamma_j\left( \frac{1}{\eta_{j,0}}+\sum_{k\ne j}\frac{\gamma_k}{\eta_{j,k}(\eta_{j,k}+\gamma_k)} \right)
          (1+\eta_{j,0}^n)\\
          &\leq 2n\tilde{M}C_j\gamma_j(1+\eta_{j,0}^n)\left( 1+\frac{1}{\eta_{j,0}}+
          \sum_{k\ne j}\frac{\gamma_k}{\eta_{j,k}(\eta_{j,k}+\gamma_k)} \right).
        \end{align*}

        Finally, we collect the above estimates and show $\Xi$ is Lipschitz.
        \begin{align*}
          \norm{\Xi(\varphi)-\Xi(\tilde{\varphi})}&=
          \sup_{j\in J}\gamma_j^{1/2}(1+\eta_{j,0}^n)^{1/2}\norm{\Xi_j(\varphi)-\Xi_j(\tilde{\varphi})}\\
          &\leq\sup_{j\in J}\gamma_j^{1/2}(1+\eta_{j,0}^n)^{1/2}\sum_{k\in J}\norm{\frac{\partial \Xi_j}{\partial \varphi_k}}_\infty \norm{\varphi_k-\tilde{\varphi}_k}\\
          &\leq \norm{\varphi-\tilde{\varphi}}\sup_{j\in J}\sum_{k\in J}\gamma_j^{1/2}(1+\eta_{j,0}^n)^{1/2}\gamma_{k}^{-1/2}(1+\eta_{k,0}^n)^{-1/2}
          \norm{\frac{\partial \Xi_j}{\partial \varphi_k}}_\infty\\
          &\leq \norm{\varphi-\tilde{\varphi}}\left(\sup_{j\in J}\gamma_j^{1/2}(1+\eta_{j,0}^n)^{1/2}\sum_{k\ne j}\gamma_{k}^{-1/2}(1+\eta_{k,0}^n)^{-1/2}
          \norm{\frac{\partial \Xi_j}{\partial \varphi_k}}_\infty
          +\sup_{j\in J}\norm{\frac{\partial \Xi_j}{\partial \varphi_j}}_\infty\right).
        \end{align*}
        We examine the sum over $k\ne j$ first.
        Using our estimates on $\norm{\frac{\partial \Xi_j}{\partial \varphi_k}}_\infty$ for
        $k\ne j$
        computed above, we see:
        \begin{align*}
          \sup_{j\in J}&\gamma_j^{1/2}(1+\eta_{j,0}^n)^{1/2}\sum_{k\ne j}\gamma_{k}^{-1/2}(1+\eta_{k,0}^n)^{-1/2}
          \norm{\frac{\partial \Xi_j}{\partial \varphi_k}}_\infty\\
          &\leq \tilde{M}\sup_{j\in J}C_j(1+\eta_{j,0}^n)^{3/2}\sum_{k\ne j}\frac{\gamma_{k}^{1/2}\gamma_j^{1/2}}{\eta_{j,k}(1+\eta_{k,0}^n)^{1/2}}\\
          &\qquad+\tilde{M}\sup_{j\in J}C_j\gamma_j^{1/2}(1+\eta_{j,0}^n)^{3/2}\sum_{k\ne j}\gamma_{k}^{1/2}(1+\eta_{k,0}^n)^{1/2}\\
          &\leq \tilde{M}\sup_{j\in J}C_j(1+\eta_{j,0}^n)^{3/2}\sum_{k\ne j}\frac{\gamma_{k}^{1/2}\gamma_j^{1/2}}{\eta_{j,k}(1+\eta_{k,0}^n)^{1/2}}\\
          &\qquad+\tilde{M}\left(\sum_{k\in J}\gamma_{k}^{1/2}(1+\eta_{k,0}^n)^{1/2}\right)\sup_{j\in J}C_j\gamma_j^{1/2}(1+\eta_{j,0}^n)^{3/2}\\
          &=:L_1
        \end{align*}
        where $L_1<\infty$ by assumption \eqref{eq:Craig2} on the first term, and \eqref{eq:Craig1} and
        \eqref{eq:Craig4} on the second.

        For the $k=j$ term, we take a supremum over our estimate on $\frac{\partial \Xi_j}{\partial\varphi_j}$,
        bounding as follows,
        \begin{align*}
          \sup_{j\in J } \left|\frac{\partial \Xi_j}{\partial \varphi_j}\right|&\leq
           2n\tilde{M}\sup_{j\in J}C_j\gamma_j(1+\eta_{j,0}^n)\left( 1+\frac{1}{\eta_{j,0}}+
          \sum_{k\ne j}\frac{\gamma_k}{\eta_{j,k}(\eta_{j,k}+\gamma_k)} \right)\\
          &\leq 2n\tilde{M}\left(\sup_{j\in J}C_j\gamma_j(1+\eta_{j,0}^n)+\sup_{j\in J}\frac{\gamma_j(1+\eta_{j,0}^n)C_j}{\eta_{j,0}}
          +\sup_{j\in J}C_j\gamma_j(1+\eta_{j,0}^n)\sum_{k\ne j}\frac{\gamma_k}{\eta_{j,k}(\eta_{j,k}+\gamma_k)}\right)\\
          &\leq 2n\tilde{M}\left(\sup_{j\in J}C_j\gamma_j(1+\eta_{j,0}^n)+\sup_{j\in J}\frac{\gamma_j(1+\eta_{j,0}^n)C_j}{\eta_{j,0}}
          +\sup_{j\in J}C_j(1+\eta_{j,0}^n)\sum_{k\ne j}{\left(\frac{\gamma_k^{1/2}\gamma_j^{1/2}}{\eta_{j,k}}\right)\!}^{\!2} \right)\\
          &=:L_2.
        \end{align*}
        In the above, the first term in $L_2$ is finite by \eqref{eq:Craig4}, once we note
        that \eqref{eq:Craig1} implies $\gamma_j<1$ for all but finitely many indices $j\in J$, so that
        for these indices, $\gamma_j<\gamma_j^{1/2}$.
        For the second term, we make use of \eqref{eq:Craig3},
        and for the final term, \eqref{eq:Craig2} suffices.

        In summary,
        \begin{align*}
          \norm{\Xi(\varphi)-\Xi(\tilde{\varphi})}\leq (L_1+L_2)\norm{\varphi-\tilde{\varphi}}
        \end{align*}
        and the vector field is Lipschitz.
        \end{proof}

  The proof of the remainder of Theorem~\ref{thm:Dubrovin}
  is now nothing more than collecting our results.

  \begin{proof}[Proof of Theorem~\ref{thm:Dubrovin}]
  Let $V:\R\to \R$ be such that $H_V$ is reflectionless and
  $\sigma(H_V)=S$, where $S$ obeys \eqref{eq:Craig1}, \eqref{eq:Craig2},
  \eqref{eq:Craig3}, and \eqref{eq:Craig4}. Associate to $V$ the angular variables
  $(\tilde{\varphi}_{j}(x))_{j\in J}\in \cR(S)$.

  Let $q(x,t)$ be a classical
  solution to \eqref{eq:kdvn},
  \eqref{eq:initial} satisfying boundedness condition \eqref{21jul3}.
  By Cor.~\ref{cor:timeinvariance} and Prop.~\ref{reflectionlessness}, we may associate
  angular variables $(\varphi_j(x,t))_{j\in J}\in \cR(S)$ at any time $t\in [0,T]$.
  Using the trace formula \eqref{eq:trace9}, this allows us to write
  \begin{align*}
  q(x,t)=\underline{E}+\sum_{j\in J} \left( E_{j}^++E_{j}^--2\mu_j(x,t) \right)
  \end{align*}

  By Prop.~\ref{lem:varphievolution2}, $\varphi_j(x,t)$ satisfies
  \begin{align*}
  \partial_t\varphi_j(x,t)=\Xi_j(\varphi)
  \end{align*}
  where $\Xi$ is Lipschitz by Lemma~\ref{lem:Lipschitz}. Thus, setting the initial
  condition
  $(\varphi_j(x,0))_{j\in J}=(\tilde{\varphi}_{j}(x))_{j\in J}$ ensures the $\mu_j(x,t)$
  in the above
  trace formula representation,
  and thus the solution $q(x,t)$,
  is unique.
  \end{proof}

\section{Application to Small Quasi-periodic initial data with diophantine frequency}

  In this section, we will show our uniqueness result applies to small quasi-periodic intial
  data in order to prove Theorem~\ref{thm:app}.

  In the setting of Theorem~\ref{thm:app} it is natural to label the gaps by the values of the rotation number $\lvert m\cdot \omega\rvert$, or more concisely, by $m\in \Z^\nu$, with the labels $m$ and $-m$ corresponding
  to the same gap, while $m=0$ corresponds to the bottom of the spectrum.
  For $\epsilon < \epsilon_0(a_0,b_0,\kappa_0)$, in \cite{DGL3,BDGL18}, the following bounds were established:
  \begin{enumerate}
  \item For every $m\in \Z^\nu\setminus \{ 0\}$,
  \begin{align}
  \label{ineq:gammaexponential}
  \gamma_m<2\epsilon \exp\left( -\frac{\kappa_0}{2}|m| \right).
  \end{align}
  \item For every $m\in \Z^\nu\setminus \{ 0\}$,
  \begin{align}
  \label{ineq:etapolynomial}
  \eta_{m,0}\leq c|m|^2
  \end{align}
  for a constant $c$ depending only on $\epsilon$ and $\omega$.
  \item For every $m\in \Z^\nu\setminus \{ 0\}$, $n\in \Z^\nu$, $m\ne n$ and $|m|\geq |n|$,
  \begin{align}
  \label{ineq:etadistance}
  \eta_{m,n}\geq a|m|^{-b}
  \end{align}
  for constants $a,b>0$ depending only on $\omega,\epsilon,\kappa_0, \nu$.
  \end{enumerate}
  As a consequence \cite{BDGL18}, the constants $C_m$ defined in \eqref{Cjdefn} obey
  \begin{align}
  \label{ineq:C_mlog}
  C_m\leq F\exp\left(F\log|m|\log\log|m|\right)
  \end{align}
  for $F=F(a,b,\kappa_0,\nu,\omega)$.

  Our goal is to use the above estimates to show that the spectrum of $H_V$, for $V$ of the form \eqref{eq:Vseries}, obeys the Craig-type conditions  \eqref{eq:Craig1},\eqref{eq:Craig2},
  \eqref{eq:Craig3}, \eqref{eq:Craig4}, and therefore that Theorem~\ref{thm:Dubrovin} applies. This will prove Theorem~\ref{thm:app}.

  \begin{lem}
  \label{lem:craigforapp}
  For $V$ of the form \eqref{eq:Vseries}, and satisfying the Diophantine condition
  \eqref{eq:Diophantine}, conditions
  \eqref{eq:Craig1},\eqref{eq:Craig2},
  \eqref{eq:Craig3}, \eqref{eq:Craig4} are satisfied for any positive integer $n$.
  \end{lem}
  \begin{proof}
  For \eqref{eq:Craig1}, we combine \eqref{ineq:gammaexponential} and \eqref{ineq:etapolynomial} to find
  \begin{align*}
  \sum_{m\in \Z^{\nu}\setminus \{ 0\}}\gamma_{m}^{1/2}(1+\eta_{m,0}^n)^{1/2}< \sqrt{2\epsilon}
  \sum_{m\in \Z^{\nu}\setminus \{ 0\}}\exp\left(-\frac{\kappa_0|m|}{4} \right)(1+c^n|m|^{2n})^{1/2}<\infty.
  \end{align*}
  For \eqref{eq:Craig2}, we use \eqref{ineq:etadistance} to bound the sum;
  \begin{align*}
  \sum_{k\ne m}\frac{\gamma_k^{1/2}}{\eta_{m,k}}&=
  \sum_{|k|> |m|}\frac{\gamma_k^{1/2}}{\eta_{m,k}}+\sum_{\substack{|k|<|m| }}\frac{\gamma_k^{1/2}}{\eta_{m,k}}\\
  &\leq \frac{2\epsilon}{a}\left(\sum_{|k|> |m|}|k|^b\exp\left(-\frac{\kappa_0|k|}{4}\right)+
  |m|^b\sum_{\substack{|k|<|m| }}\exp\left(-\frac{\kappa_0|k|}{4}\right)\right)\\
  &\leq \frac{2\epsilon}{a}\left(\sum_{k\ne m}|k|^b\exp\left(-\frac{\kappa_0|k|}{4}\right) \right)(1+|m|^{b}).
  \end{align*}
  And from \eqref{ineq:C_mlog}, \eqref{ineq:etapolynomial}, we have for $m\in \Z^\nu\setminus \{ 0\}$,
  \begin{align*}
  C_m\gamma_{m}^{1/2}(1+\eta_{m,0}^n)^{3/2}
  &\leq
  F\exp\left(F\log|m|\log\log|m|-\frac{\kappa_0|m|}{4}\right)(1+c^n|m|^{2n})^{3/2}.
  \end{align*}
  And by taking the product of these estimates we may conclude
  \begin{align*}
    \sup_{m\neq 0}C_m\gamma_m^{1/2}(1+\eta_{m,0}^n)^{3/2}
    \sum_{k\ne m}\frac{\gamma_k^{1/2}}{\eta_{m,k}}<\infty.
  \end{align*}
  For \eqref{eq:Craig3}, it suffices to note for $m\in \Z^\nu\setminus \{ 0\}$,
  \begin{align*}
  \frac{\gamma_m(1+\eta_{m,0}^n)C_m}{\eta_{m,0}}\leq
  \frac{2\epsilon}{a}F\exp\left(F\log|m|\log\log|m|-\frac{\kappa_0}{2}|m|\right)
  (1+c^n|m|^{2n})|m|^b
  \end{align*}
  and so $\sup_{m\neq 0}\frac{\gamma_m(1+\eta_{m,0}^n)C_m}{\eta_{m,0}}<\infty$.
  \eqref{eq:Craig4} follows immediately from the estimate
  \begin{align*}
  \gamma_{m}^{1/2}(1+\eta_{m,0}^n)^{3/2}C_m\leq
  \sqrt{2\epsilon} F\exp\left(F\log|m|\log\log|m|-\frac{\kappa_0}{4}|m|\right)
  (1+c^n|m|^{2n})^{3/2}
  \end{align*}
  for $m\in \Z^\nu\setminus \{ 0\}$.
  \end{proof}

  \begin{proof}[Proof of Theorem \ref{thm:app}]
  By \cite{DG1}, $\sigma(H_V)=\sigma_{\ac}(H_V)$. Thus, by Kotani theory, $H_V$ is
  reflectionless \cite{Kotani}.
  Furthermore, by Lemma~\ref{lem:craigforapp}, conditions
  \eqref{eq:Craig1},\eqref{eq:Craig2},
  \eqref{eq:Craig3}, and \eqref{eq:Craig4} are satisfied.
  Thus we have, by Theorem
  \ref{thm:Dubrovin}, uniqueness of any solution $q(x,t)$ to \eqref{eq:kdvn}, \eqref{eq:initial},
  provided $q,\partial_{x}^{2n}q \in L^\infty(\R\times [0,T])$.
  \end{proof}

\def\cprime{$'$}

\end{document}